\documentclass[11pt,a4paper]{article}

\usepackage[utf8]{inputenc}
\usepackage[T1]{fontenc}
\usepackage{lmodern}
\usepackage{microtype}
\usepackage{geometry}
\usepackage[dvipsnames]{xcolor}
\usepackage{amsmath,amssymb,amsthm,mathtools}
\usepackage{bm}
\usepackage{enumitem}
\usepackage{faktor}
\usepackage[colorlinks=true, linkcolor = MidnightBlue, citecolor= OliveGreen]{hyperref}
\usepackage[capitalise,nameinlink]{cleveref}
\usepackage{graphicx}
\usepackage{xcolor}
\usepackage[numbers]{natbib}
\usepackage{algorithm}
\usepackage{algpseudocode}
\usepackage{subcaption}

\title{Retraction based regression methods on Riemannian manifolds}
\author{Estefanía Loayza-Romero$^{1,}$\footnote{\texttt{estefania.loayza-romero@strath.ac.uk}} , Benedikt Sibum$^{2,}$\footnote{\texttt{sibumb@hsu-hh.de}} \, and Kathrin Welker$^{2,}$\footnote{\texttt{welker@hsu-hh.de}} }

\date{\small $^1$University of Strathclyde, 26 Richmond St, Glasgow G1 1XQ, United Kingdom\\[4pt]\small $^2$Helmut-Schmidt-University / University of the Federal Armed Forces Hamburg,\\ Holstenhofweg~85, 22043 Hamburg, Germany}

\theoremstyle{plain}
\newtheorem{theorem}{Theorem}[section]
\newtheorem{proposition}[theorem]{Proposition}

\theoremstyle{definition}

\newcommand{\RR}{\mathbb{R}}

\newcommand{\set}[2]{\left\{#1 \;\middle|\; #2\right\}}		

\newcommand{\norm}[1]{\left|\left| #1 \right|\right|}

\DeclareMathOperator{\id}{id}
\DeclareMathOperator{\grad}{grad}
\DeclareMathOperator{\Hess}{Hess}
\DeclareMathOperator{\dist}{dist}
\DeclareMathOperator{\Ker}{Ker}
\DeclareMathOperator{\sgn}{sgn}
\DeclareMathOperator{\diag}{diag}

\renewcommand{\d}{\operatorname{d}}
\newcommand{\hor}{\operatorname{hor}}
\newcommand{\ver}{\operatorname{ver}}
\newcommand{\eucl}{\operatorname{eucl}}

\begin{document}

\maketitle

\begin{abstract}
	Geodesic regression generalizes classical regression models to manifold-valued data by replacing affine models in Euclidean spaces with geodesic models on Riemannian manifolds. In this paper, we set up a framework for regression based on retractions instead of the Riemannian exponential map and its corresponding retraction-based distance. The associated optimization problem is posed on a subset of the tangent bundle which is why we additionally construct retractions on the tangent bundle induced by retractions on the underlying manifold. Our approach yields a more flexible formulation which is applicable beyond settings where the exponential map can be computed efficiently. As a proof of concept, we apply the developed framework to the $(n-1)$-dimensional $p$-norm sphere using the retraction by normalization to define the regression problem. The resulting optimization problem is solved using the Riemannian steepest descent method. 
\end{abstract}




\section{Introduction}

Regression problems are a fundamental class of methods in statistics and machine learning. They allow us to determine how a measured variable is related to one or more independent variables. Classical regression theory is formulated in Euclidean spaces, where linear structures allow the use of affine models, standard optimization techniques and even explicit expressions can be obtained. However, in recent years, it has been acknowledge that in many applications data naturally belongs to nonlinear spaces that possess a non-Euclidean geometric structure. Such manifold-valued data arise, for instance, in robotics \cite{wang2019}, where rotations are represented by orthogonal matrices, in diffusion tensor imaging \cite{pennec2006, dryden2009}, where data is modeled by symmetric positive definite matrices, or in computer vision \cite{turaga2016, cherian2016} and shape analysis, \cite{ring2012, schulz2016}, where shapes are represented as points on infinite-dimensional shape manifolds. In these settings, classical Euclidean regression models fail to take into account the intrinsic geometry of the data, making the development of statistical methods on Riemannian manifolds necessary. One of the first intrinsic generalizations of linear regression to Riemannian manifolds using the Riemannian geometry is geodesic regression, introduced for example in \cite{fletcher2013, niethammer2011}, which laid the groundwork for further works such as \cite{shin2022, hong2012, hinkle2012, fishbaugh2014}. In this framework, Euclidean straight lines are replaced by geodesics, while the intercept and slope are represented by elements of the tangent bundle. A main computational challenge of geodesic regression is the repeated evaluation of the Riemannian exponential map, which requires the solution of the geodesic equation, i.e.,  of a second-order ordinary different equation. Depending on the underlying manifold and metric structure, this can lead to very expensive computational costs. This issue happens for example in shape optimization and shape analysis \cite{schulz2016}, where the optimization domain is often modeled as an infinite-dimensional Riemannian manifold, e.g., \cite{michor2006, bauer2014a}.
There are Riemannian metrics such as the Steklov–Poincaré metric \cite{schulz2016}, for which the geodesic equation has not been derived analytically. Since retractions approximate the exponential map and preserve the local geometric structure required for optimization on manifolds, they are a computationally efficient alternative \cite{absil2008}. 
Our paper is motivated by this observation. In this paper, we propose a generalized regression framework on Riemannian manifolds in which the exponential map is replaced by retractions. 
The resulting optimization problem is formulated on a subset of the tangent bundle, and thus, we investigate the construction of retractions on tangent bundles induced by retractions on the underlying manifold. 

The paper is organized as follows. \Cref{sec:Basics} gives a short overview of the required concepts on differential geometry and sets up our notation and terminology. We then formulate the regression problem on general Riemannian manifolds using retractions instead of the exponential map in \Cref{sec:main}. \Cref{prop:well-posed} studies the well-possedness of the proposed regression problem. Furthermore, we provide a formula for the gradient of the objective function for general Riemannian manifolds as well as for embedded submanifolds. In addition, \Cref{theorem: retraction tangent bundle} provides a way to construct retractions on the tangent bundle.
As a proof of concept, in \Cref{sec:numerics}, we apply the proposed framework to $(n-1)$-dimensional, $p$-norm spheres with the so-called retraction by normalization function and present numerical experiments. To conclude, we summarize our results in \Cref{sec:Conc}.

\section{Notation and terminology from differential geometry}
\label{sec:Basics}

In the following, we introduce notation and concepts from differential geometry required in this paper; for detailed definitions of the introduced objects, we refer to the literature \cite{lang1999, leeJeffrey2009, lee2012, lee2018, michor2008}.

\textbf{Vector bundles.}
A smooth\footnote{Throughout this paper, smooth means infinitely continuous differentiable.} vector bundle over a smooth manifold $M$ is a triple $(E, \pi, M)$, where $E$ is a smooth manifold  and $\pi\colon E\to M$ a projection. For each $p\in M$ we denote by $E_p\coloneqq \pi^{-1}(p)$ the fibre of $\pi$ over $p$. If  $E$ is a vector bundle,  $W$ a manifold and  $f\colon E\to W$, then we denote $f_p\coloneqq f\big|_{E_p}\colon E_p\to W$. If the map already has a subscript like $f_i$, we denote its fibre restriction by $(f_i)_p$. The set $\Gamma(E)$ consists of all smooth sections of the vector bundle $E$.  To emphasize the base point we will write an element $v\in E$ as a tuple $(p, v)$ or with a subscript $v_p$ with $p\in M$ and $v\in E_p$. If a vector bundle homomorphism $f\colon E\to F$ is given, then we denote $\Ker(f)\coloneqq \bigcup_{p\in M} \Ker(f_p)$, where $\Ker(f_p)\coloneqq \set{v\in E_p}{f_p(v) = 0_{T_{f(p)}N}}$. For a smooth manifold $M$ the tangent bundle is denoted by $TM$ and is itself a vector bundle $(TM, \pi_M, M)$, where we denote each fibre $(TM)_x$ by $T_xM$. Since $TM$ is a smooth manifold, we can consider the tangent bundle of $TM$ denoted by $(T^2M, \pi_{TM}, TM)$, where $T^2M = T(TM)$.

\textbf{Differential calculus on smooth manifolds.}
Let $f\colon M\to N$ be a smooth map between manifolds $M,N$, then the differential of $f$ is a vector bundle homomorphism denoted by $\textup{d}f\colon TM\to TN$. Given a suitable local frame in the domain and codomain, the local coordinate representation of $\textup{d}f$ is given by the Jacobian of $f$ denoted by $J_f$, where we suppress the frame dependency. A general curve $\gamma$ on $M$ is defined as a mapping from an open interval $I\subset\RR$ containing $0$ to $M$, i.e., $\gamma\colon I\to M$. Let $\gamma\colon I\to M$ be a smooth curve and $\partial t\in\Gamma(TI)$ the canonical basis vector field in $TI$, then we define $\dot{\gamma}(s) = (d\gamma\circ \partial t)(s)$ for $s\in I$.

\textbf{Horizontal and vertical subbundles.}
Given a linear connection $C\colon TM\times_M TM\to T^2M$ on $M$, where $\times_M$ denotes the fibre product over $M$, and the induced connector $K\colon T^2M\to TM$, the second tangent bundle can be decomposed into $T^2M\simeq HTM\oplus VTM$ according to \cite[p.~203]{michor2008}, where $VTM\coloneqq \Ker(\textup{d}\pi_{M})$ is called the vertical and $HTM\coloneqq \Ker(K)$ the horizontal subbundle. Therefore, each vector $\xi\in T^2M$ can be uniquely decomposed into a horizontal component $\xi^{\hor}\coloneqq \textup{d}\pi_M(\xi)\in TM$ and a vertical one $\xi^{\ver}\coloneqq K(\xi)\in TM$ at the point $(x, v)\in TM$.
So we can identify $\xi\in TM$ with the triple $(\pi_{TM}(\xi), \xi^{\hor}, \xi^{\ver})$. The vertical component corresponds to the change of $v$ within the fibre whereas the horizontal component corresponds to the change of the base point $x$. The linear connection tells us how we can identify certain tangent vectors if we only want to change the base point.

Let $N$ be another smooth manifold and $s\colon N\to TM$ a smooth mapping, the linear connection $C$ induces the covariant derivative of $s$ in the direction of $X\in\Gamma(TN)$ which is defined by $\nabla_X s\coloneqq K\circ \textup{d}s\circ X\colon N\to TN\to T^2M\to TM$. The covariant derivative can be extended on tensor bundles of any order; we denote the extension of the covariant derivative also by  $\nabla$. The covariant derivative reduces the information of the full differential $\textup{d}s$ only by a variation of the base point through $X$ and projecting onto the vertical change through $K$. With the help of the covariant derivative we can represent horizontal and vertical components of an element in $T^2M$ more specifically, which we  use in \cref{theorem: retraction tangent bundle} to proof the local rigidity condition of our constructed retraction on $T^2M$.
Therefore, if $\gamma\colon I\to TM$  is a smooth curve on $TM$, we can represent the curve $\gamma(t) = (u(t), V(t))$ with $u\colon I\to M$ as a smooth curve in $M$ and $V\colon I\to TM$ as a smooth vector field along $u$. 
If $t_0\in I$, then the horizontal component of $\dot{\gamma}(t_0)$ corresponds to
$
(\textup{d}\pi_M)_{\gamma(t_0)}(\dot{\gamma}(t_0))
= \textup{d}(\pi_M\circ \gamma))_{t_0}(\partial t(t_0))
= \textup{d}u_{t_0}(\partial t(t_0))
= \dot{u}(t_0)
$,
whereas the vertical change is given by
$
K_{\gamma(t_0)}(\dot{\gamma}(t_0))
= (K_{\gamma(t_0)}\circ \textup{d}\gamma \circ \partial t)(t_0)
= \nabla_{\partial t} V(t_0).
$
Thus, we can identify $\dot{\gamma}(t_0)$ with
\begin{equation}
	\label{eq: ver_hor_identification}
	(\pi_{TM}(\dot{\gamma}(t_0)), \dot{\gamma}(t_0)^{\hor}, \dot{\gamma}(s)^{\ver})
	= (\gamma(t_0), \textup{d}\pi_M(\dot{\gamma}(t_0)), K(\dot{\gamma}(t_0)))
	= (\gamma(t_0), \dot{u}(t_0), \nabla_{\partial t} V(t_0)).
\end{equation}

\textbf{Riemannian manifolds.}
We denote a Riemannian manifold by a tuple $(M, g)$ where $M$ is a smooth manifold and $g$ is the Riemannian metric. Throughout this paper, we consider a Riemannian manifold $(M, g)$ equipped with the Levi-Civita connection associated with $g$. According to \cite{sasaki1958}, the Riemannian metric $g$ on $M$ induces a Riemannian metric on the tangent bundle $TM$, the so-called Sasaki metric, which we denote by $g^S$. We will also need the concept of the exponential map, and its approximation, the so-called retraction. We denote the exponential map restricted at $x$ by $\exp_{x}\colon T_{x}M \to M,\,v\mapsto \exp_{x}(v)$, which assigns to every tangent vector $v\in T_xM$ the value $\tilde{\gamma}(1)$ of the geodesic $\tilde{\gamma}\colon [0,1]\to M$ satisfying $\tilde{\gamma}(0) = x$ and $\dot{\tilde{\gamma}}(0) = v.$ A retraction is a mapping $R\colon TM\to M$, where $R_x\colon T_x M\to M$ is the restriction of $R$ to $T_xM$ satisfying $R_x(0_x)=x$ and the local rigidity condition $(\textup{d}R_x)_{(x, 0)}=\text{id}_{T_xM}$. Here, we use the identification $T_{(x, 0)}(T_xM)\simeq T_xM$. In other words, a retraction is a local approximation of the exponential map. In general, we will identify the tangent space of a general vector space with itself.

\textbf{Embedded submanifold theory.}
For convenience of the reader, we recall some theoretical results regarding submanifolds to prepare the reader to \cref{sec:numerics}, in which we investigate the behaviour of our proposed framework on Riemannian submanifolds of an Euclidean space $\mathcal{E}$. First, we denote the scalar product and therefore the Riemannian metric of $\mathcal{E}$ just by $\left<\cdot, \cdot\right>$ and the covariant derivative induced by the Levi-Civita connection by $\nabla^{\eucl}$. If $M\subset \mathcal{E}$ is a smooth embedded submanifold of $\mathcal{E}$, meaning that it exists a smooth embedding $\iota\colon M\to \mathcal{E}$, then we consider each tangent space $T_xM\subset T_x\mathcal{E}\simeq \mathcal{E}$ as a subspace. For $M$ to be a Riemannian submanifold of $\mathcal{E}$, $M$ has to be additionally equipped with the Riemannian metric defined as the pullback of the Euclidean and denoted by $\iota^*\left<\cdot, \cdot \right>$. Therefore, we can use the differential operators from $\mathcal{E}$ to obtain derivatives of functions defined on $M$ like gradients of smooth functions or covariant derivatives of tensor fields on $M$. To this end, we will assume smooth extensions of these maps to an open environment of $\mathcal{E}$ including $M$. The existence of such smooth extensions is guaranteed by \cite[Lem.~5.34, p.~115]{lee2012} which we denote by the same symbol for simplicity. According to \cite[Eq.~(5. 4), p.~87]{boumal2023}, \cite[Thm.~5.9, p.~92]{boumal2023} and \cite[Prop.~3.61, p.~44]{boumal2023} the covariant derivative of the Levi-Civita connection of $M$ and the gradient are given by
\begin{align}
	\nabla_u V &= P_x(\nabla_u^{\eucl}(V)),
	\quad
	V\in\Gamma(TM), u\in T_xM,
	\label{eq: embedded covariant derivative}\\
	\grad(f) &= P_x(\grad^{\eucl}(f)), 
	\quad
	f\in C^{\infty}(M, \RR),
	\label{eq: general embedded gradient}
\end{align}
where $P_x\colon \mathcal{E}\to T_xM$ is the orthogonal projection. Throughout the paper, the symbol $P$ is reserved for an orthogonal projection. 

Additionally, according to \cite[Prop.~5.16, p.~106]{lee2012} every smooth embedded submanifold $M\subset\mathcal{E}$ can  be represented locally as a regular level set of a local defining function $\phi\colon U\subset\mathcal{E}\to \RR^{k}$. In this context, the tangent spaces at $x\in M\cap U$ can be expressed by $T_xM = \Ker(\textup{d}\phi_x)\subset\mathcal{E}$; see \cite[Prop.~5.38, p.~117]{lee2012}. It is also possible to specify the orthogonal projection $P_x\colon\mathcal{E}\to T_xM$ regarding the standard coordinate frame of $\mathcal{E}$ by 
\begin{align}
	\label{map: coordinate projection}
	P_x(v) = \left(E_n - J_{\phi}(x)^{\top}(J_{\phi}(x)J_{\phi}(x)^{\top})^{-1}J_{\phi}(x)\right)\cdot v
\end{align}
 for $v\in \mathcal{E}$, where $E_n$ is the $n$-dimensional identity matrix; see \cite[Eq.~(5.13.3), p.~430]{meyer2023}.

\section{Retraction-induced regression framework} 
\label{sec:main}

Let $(M, g)$ be a Riemannian manifold and $\mathcal{D}\coloneqq \left\{(t_1, y_1), \ldots (t_m, y_m)\right\}\subset \RR\times M$ be a finite collection of points call data set. The regression problem can be formulated as the search after a curve $c\colon \RR\to M$ out of a predefined set of curves $\mathcal{C}$ which fits the data the best regarding a cost function $\mathcal{F}\colon \mathcal{C}\to[0, \infty)$. Therefore, the goal is to solve the minimization problem $\min\limits_{c\in\mathcal{C}} \mathcal{F}(c)$. In this paper, we refer to $\mathcal{C}$ as the model restriction and $c\in\mathcal{C}$ as a model.

The geodesic regression problem formulated in \cite{fletcher2013} considers as $\mathcal{C}$ the set containing every geodesic in $M$ characterized by its starting conditions which are elements of the tangent bundle $TM$. In other words, the model is parametrized by an element of the tangent bundle, which is specified by $c_{\exp, (x, v)}\colon \RR \subset I \to M,\ t\mapsto \exp(x, tv)$ for $(x, v)\in TM$. The cost function is thus given by the manifold formulated version of the sum of squared residuals 
\begin{equation}
\label{eq:sum_squares}
	\mathcal{F}_{\exp}\colon TM\to [0, \infty),\
	(x, v)\mapsto \frac{1}{2} \sum\limits_{i = 1}^m \dist_{y_i}\left(c_{\exp, (x, v)}(t_i)\right)^2,
\end{equation}
where $\dist(y,x)$ denotes the Riemannian distance on $M$ between $x,y$ and we consider the mapping $\dist_{y}\coloneqq \dist(y, \cdot)\colon M\to [0, \infty)$. 
Under these considerations, the geodesic regression problem is then defined as
\begin{equation}
\label{eq:regression_main}
\min_{(x,v)\in TM} \mathcal{F}_{\exp}(x,v).
\end{equation}
Notice that both, the model and the Riemannian distance rely on the exponential map.

Our goal is to define a function like $\mathcal{F}$ from \eqref{eq:sum_squares} which does not rely on the exponential map, its inverse or the Riemannian distance. In other words, the model is given by $c_{R(x, v)}\colon \RR\supset I\to M,\ t\mapsto R(x, tv)$, where $R\colon TM\to M$ is a retraction as defined in \cref{sec:Basics}. Associated with the retraction, we consider the retraction-based distance $\dist^R$ (see e.g.~\cite{bergmann2025}) given by:
\begin{align*}
		\dist_x^R\colon D_x\to [0, \infty),\ 
		y\mapsto \sqrt{g_x\left(R_x^{-1}(y), R_x^{-1}(y))\right)}.
\end{align*}
We recall that since $\left(\textup{d}R_x\right)_{(0, x)}$ is surjective for every $x\in M$, the inverse function theorem implies that there exist open sets $U_x\subset T_xM$ and $D_x\subset M$ so that $R_x\colon U_x\to D_x$ is invertible, and thus $\dist_x^R$ is well defined for each $x\in M$. Since the equation $\dist(x, y) = \sqrt{g_x\left(\exp_x^{-1}(y), \exp_x^{-1}(y)\right)}$ holds  for every $(x, y)\in M\times W_x$, where $W_x$ is the complement of the cut locus of $x$, the retraction-based distance locally approximates the Riemannian distance.

Now, we are able to present the proposed objective function associated to the regression problem given in~\eqref{eq:regression_main}.
\begin{proposition}
	\label{prop:well-posed}
	Let $(M, g)$ be a Riemannian manifold and $\mathcal{D} = \left\{(t_1, y_1), \ldots, (t_m, y_m)\right\}\subset \RR\times M$ be a finite data set.  Furthermore, let $R\colon TM\to M$ be a retraction and $l_t\colon TM\to TM,\  (x, v)\mapsto (x, tv)$ a vector bundle homomorphism for each $t\in\RR$. We define $V_{y_i}\coloneqq (R\circ l_{t_i})^{-1}(D_{y_i})\subset TM$ and $V\coloneqq \bigcap\limits_{i=1}^m V_{y_i}$. The sets $D_{y_i}$ are chosen such that $R_{y_i} \colon U_{y_i} \to D_{y_i}$ is invertible.
	Then, the mapping defined by 
	\begin{align}
		\label{map: objective function}
		\mathcal{F}_R\colon V\to [0, \infty),\
		(x, v)\mapsto \frac{1}{2} \sum\limits_{i=1}^m \dist_{y_i}^R\left(c_{R, (x, v)}(t_i)\right)^2
	\end{align}
	is well-posed. 
\end{proposition}
\begin{proof}
	Since the retraction-based distance is just defined locally, we have to ensure that the model output $c_{R, (x, v)}(t_i)$ lies in $D_{y_i}$ for each $(x, v)\in V$ and $i=1, \ldots, m$. Therefore, let $i\in\{1, \ldots, m\}$ and $(x, v)\in V$ be arbitrary. According to the definition of $V$ we have especially that $(x,v)\in V_{y_i}$. Thus, the equation $c_{R, (x, v)}(t_i) = R(x, t_iv) = R\left(l_{t_i}(x, v)\right)\in D_{y_i}$ holds for every $i= 1, \ldots, m$ which concludes the proof. 
\end{proof}
The retraction-based regression problem is then defined through problem~\eqref{eq:regression_main} with $\mathcal{F} = \mathcal{F}_R$, where the optimization domain is given by $V\subset TM$.

As already noted in \cite{fletcher2013}, this problem does not yield an analytic solution, even in the geodesic case; and thus, to find a local minimum of $\mathcal{F}_R$, we use the Riemannian steepest descent method. To this end we compute the gradient of $\mathcal{F}_R$. Relying on the linearity of the derivatives, we compute the gradient of only one summand of $\mathcal{F}_R$. Let $i\in\{1, \ldots, n\}$ be fixed but arbitrary, we consider the functions
\begin{align}
	\label{map: auxillary maps}
	f_i & \colon V_{y_i}\to [0, \infty),  \, (x,v) \mapsto \left(\dist_{y_i}^{R}\circ R\circ l_{t_i}(x,v)\right)^2,\\ 
	\nonumber
	h_i& \colon V_{y_i}\to T_{y_i}M, \, (x,v) \mapsto R_{y_i}^{-1}\circ R\circ l_{t_i}(x,v),
\end{align}
which leads to $\dist_{y_i}^{R}\left(c_{R,(x, v)}(t_i)\right)^2 = f_i(x, v) = g_{y_i}\left(h_i(x, v), h_i(x, v)\right)$ for every $(x, v)\in V_{y_i}$. 
Since each $V_{y_i}$ is an open subset of $TM$, we have $T_{(x, v)}V_{y_i}\simeq T_{(x, v)}TM$.
For $\delta\in T_{(x, v)}V_{y_i}\simeq T_{(x, v)}TM$ we have
\begin{equation}
	\label{eq: gradient}
	(\textup{d}f_i)_{(x, v)}(\delta)
	= 2 g_{y_i}\left(\textup{d}h_i)_{(x, v)}(\delta), h_i(x, v)\right) 
	= g^S_{(x, v)}\left(\delta, 2(\textup{d}h_i)_{(x, v)}^*\left(h_i(x, v)\right)\right),
\end{equation}
where $(\textup{d}h_i)_{(x, v)}^*\colon T_{y_i}M\to T_{(x, v)}TM$ is the adjoint map of $(\textup{d}h_i)_{(x, v)}\colon T_{(x, v)}TM\to T_{y_i}M$. For the codomain of the adjoint map we used the identification $T(T_{y_i}M)\simeq T_{y_i}M$. The second equation in (\ref{eq: gradient}) gives $\grad f_i(x, v) = (\textup{d}h_i)_{(x, v)}^*(h_i(x, v))$ for every $(x, v)\in V_{y_i}$. This allows us to conclude that the gradient of the cost function is given by 
\begin{equation}
	\label{eq: gradient objective fct}
	\grad \mathcal{F}_R (x, v)
	= \sum\limits_{i = 1}^m (\textup{d}h_i)_{(x, v)}^*(h_i(x, v))\quad , \, (x, v)\in V.
\end{equation}
The formula for the gradient of the objective function given in \eqref{eq: gradient objective fct} coincides with that one presented in \cite{fletcher2013} which becomes clear when the identity $(\textup{d}\exp^{-1}_y)_z^*\left(\exp^{-1}_y(z)\right) = - \exp^{-1}_z(y)$ is used.

\paragraph{Embedded submanifolds in $\RR^n$.}
If the manifold $M$ is a Riemannian submanifold of $\RR^n$, the tangent bundle $TM$ is a submanifold of $\RR^n\times \RR^n$ since every local defining function $\phi\colon \RR^n\supset U\to \RR^k$ of $M$ induces a local defining function $\psi\colon U\times\RR^n\to \RR^{2k},\ (x, v)\mapsto (\phi(x), \textup{d}\phi_x(v))$ of $TM$; see, e.g., \cite[Thm~3.43, p.~37-38]{boumal2023}.
Then, we can represent the gradient of $\mathcal{F}_R$  in standard coordinates of $\RR^n\times \RR^n$ thanks to (\ref{map: coordinate projection}):
\begin{align}
	&\grad \mathcal{F}_R(x, v)\nonumber\\
	&= P_{(x, v)}\left(\grad^{\eucl}\mathcal{F}_R(x, v)\right) \nonumber \\
	&= \sum\limits_{i=1}^n \left(E_n - J_{\psi}(x, v)^{\top}\left(J_{\psi}(x, v)J_{\psi}(x, v)^{\top}\right)^{-1}J_{\psi}(x, v)\right)\left(J_{h_i}(x, v)^{\top} h_i(x, v)\right).
	\label{eq: embedded gradient objective fct}
\end{align}
The derivative of $\psi$ at $(x, v)\in T(U\times\RR^n)$ in the direction of $\xi\in T_{(x, v)}T(U\times \RR^n)$ and its corresponding Jacobian is  given by
\begin{align*}
	\textup{d}\psi_{(x, v)}(\xi)
	= \begin{pmatrix}
		\textup{d}\phi_x(\xi^{\hor}) \\
		\left(\nabla^{\eucl}_{\xi^{\hor}} \textup{d}\phi\right)_x(v) + \textup{d}\phi_x\left(\xi^{\ver}\right)
	\end{pmatrix},
	\quad
	J_{\psi}(\hat{x}, \hat{v})
	= \begin{pmatrix}
		J_{\phi}(\hat{x}) & 0 \\
		\hat{v}\cdot A_{\phi}(\hat{x})& J_{\phi}(\hat{x})
	\end{pmatrix}
	\cdot \begin{pmatrix}
		\hat{\xi}^{\hor} \\
		\hat{\xi}^{\ver}	
	\end{pmatrix},
\end{align*}
where we use the hat accent to indicate the corresponding coordinate representation and denote by $A_{\phi}(\hat{x})$ the matrix representation of $\left(\nabla^{\eucl}_{\xi^{\hor}} \textup{d}\phi\right)_x\colon T_xU\to \RR^k$ in standard coordinates.

To move on the manifold $V\subset TM$ in the direction $\delta\in T^2M$ we must consider a retraction $\mathcal{R}\colon T^2M\to TM$ on the second tangent vector bundle given by the following theorem.
\begin{theorem}	
	\label{theorem: retraction tangent bundle}
	Let $M$ be a smooth manifold and $R\colon TM\to M$ a retraction.
	Moreover, let $\nabla$ and $K$ the covariant derivative and connector of a predefined connection on $M$, respectively. Furthermore, let $\mathcal{T}\colon TM\oplus TM\to TM$ be a vector transport\footnote{We refer the reader to \cite{absil2008} for the exact definition and further details.} on $M$ associated to $R$. Then, the map $\mathcal{R}\colon T^2M\to TM$ given by 
	\begin{align*}
		\mathcal{R}_{(x, v)}(\xi)
		= \left(R_x(\xi^{\hor}), \mathcal{T}_{\xi^{\hor}}(v + \xi^{\ver})\right)
		\in T_{R_x(\xi^{\hor})}M
	\end{align*}
	for $(x, v)\in TM$ and $\xi\in T_{(x, v)}TM$ is a retraction of $T^2M$.
\end{theorem}
\begin{proof}
	First, we show that $\mathcal{R}_{(x, v)}(0_{(x, v)}) = (x, v)$ holds for all $(x, v)\in TM$: We have $0_{(x, v)}^{\hor} = 0_x = 0_{(x, v)}^{\ver}$ for all $(x, v)\in TM$ and, thus, we can conclude 
	\begin{align}
		\label{eq: 1_retr_cond}
		\mathcal{R}_{(x, v)}\left(0_{(x, v)}\right)
		= \left(R_x(0_x), \mathcal{T}_{0_x}(v + 0_x)\right)
		= \left(x, \id_{T_xM}(v)\right)
		= (x, v)\quad \forall \,(x, v)\in TM.
	\end{align}
	It remains  to be shown that $\mathcal{R}$ fulfills the local rigidity condition
	\begin{align*}
		\left(\textup{d}\mathcal{R}_{(x, v)}\right)_{0_{(x, v)}}
		= \id_{T_{(x, v)}TM}\quad \forall \,(x, v)\in TM.
	\end{align*}
	Let $\xi\in T_{(x, v)}TM$. 
	As we have already seen in \cref{sec:Basics}, it suffices if the vector $(\textup{d}\mathcal{R}_{(x, v)})_{0_{(x, v)}}(\xi)$ corresponds to the triple $((x, v), \xi^{\hor}, \xi^{\ver})$. 
	From~\eqref{eq: 1_retr_cond}, we know that the first component of the triple coincides with $\pi_{TM}((\d\mathcal{R}_{(x, v)})_{0_{(x, v)}}(\xi)) = \mathcal{R}_{(x, v)}\left(0_{(x, v)}\right)$ for all $(x, v)\in TM$. Now, we compute the horizontal and vertical changes by means of equation \eqref{eq: ver_hor_identification}. To this end, we consider $c\colon I\to T_{(x, v)}TM,\ t\mapsto t\xi$. Thus, $c$ is a smooth curve with $c(0) = 0_{(x, v)}$ and $\dot{c}(0) = \xi$. Now, we consider the $TM$ valued curve $\mathcal{R}_{(x, v)}\circ c$. Then,
	\begin{align*}
		(\mathcal{R}_{(x, v)}\circ c)(t)
		= \left(R_x(c(t)^{\hor}),\ \mathcal{T}_{c(t)^{\hor}}(v + c(t)^{\ver})\right)
		= \left(R_x(t \xi^{\hor}),\ \mathcal{T}_{t\xi^{\hor}}(v + t \xi^{\ver})\right)
	\end{align*}
	for $(x, v)\in TM$ and $t\in I$. The last equality holds thanks to the linearity of $\textup{d}\pi_M$ and $K$: 
	\begin{align*}
		&c(t)^{\hor} 
		= K(c(t)) 
		= K(t \xi) 
		= t K(\xi) 
		= t \xi^{\hor}, \\
		&c(t)^{\ver}
		= \textup{d}\pi_M(c(t))
		= \textup{d}\pi_M(t \xi)
		= t \textup{d}\pi_M(\xi)
		= t \xi^{\ver}.
	\end{align*}
	Now, we calculate the horizontal and vertical changes of $(\textup{d}\mathcal{R}_{(x, v)})_{0_{(x, v)}}(\xi) = \textup{d}(\mathcal{R}_{(x, v)}\circ c)_0(\partial t(0))$. The horizontal change is given by 
	\begin{align*}
		(\textup{d}\pi_M)_{(x, v)}\left(\textup{d}(\mathcal{R}_{(x, v)}\circ c)_0(\partial t(0))\right)
		&= \textup{d}(R_x\circ c_{\hor})_0(\partial t(0))
		= (\textup{d}R_x)_{(x, 0)}\left(\dot{c}_{\hor}(0)\right) \\
		&= \id_{T_xM}(\xi^{\hor})
		= \xi^{\hor},
	\end{align*}
	where $c_{\hor}$ is the smooth curve given by $c_{\hor}(t) = c(t)^{\hor} = t \xi^{\hor}$.
	Before we calculate the vertical change we note that the map $t\mapsto \mathcal{T}_{t\xi^{\hor}}(v + t \xi^{\ver}) \in T_{R_x(t\xi^{\hor})}M$ is a smooth vector field along $\gamma_{\xi}\colon I\to M,\
	t\mapsto R_x\left(t \xi^{\hor}\right)$ which we will denote by $V$ for simplicity. So we have 
	\begin{align*}
		K_{(x, v)}\left(\textup{d}(\mathcal{R}_{(x, v)}\circ \gamma)_0(\partial t(0))\right)
		&= \nabla_{\partial t} \left(\mathcal{T}_{t\xi^{\hor}}(v + t \xi^{\ver})\right)\bigg|_{t = 0} 
		= \nabla_{\partial t} V(0) \\
		&\overset{\text{\cite{lee2018}}}{=} \lim\limits_{t\to 0} \frac{\mathcal{P}_{t, 0}^{\gamma_{\xi}}(V(t)) - V(0)}{t} 
		= \lim\limits_{t\to 0} \frac{\mathcal{P}_{t, 0}^{\gamma_{\xi}}(\mathcal{T}_{t\xi^{\hor}}(v + t \xi^{\ver}))}{t} \\
		&= \lim\limits_{t\to 0} \frac{\mathcal{P}_{t, 0}^{\gamma_{\xi}}(\mathcal{T}_{t\xi^{\hor}}(v)) + t \mathcal{P}_{t, 0}^{\gamma_{\xi}}(\mathcal{T}_{t\xi^{\hor}}(\xi^{\ver}))}{t} \\
		&= \underbrace{\lim\limits_{t\to 0} \frac{\mathcal{P}_{t, 0}^{\gamma_{\xi}}(\mathcal{T}_{t\xi^{\hor}}(v))}{t}}_{\to 0_x}
		+ \lim\limits_{t\to 0}\mathcal{P}_{t, 0}^{\gamma_{\xi}}(\mathcal{T}_{t\xi^{\hor}}(\xi^{\ver})) \\
		&= \mathcal{P}_{0, 0}^{\gamma_{\xi}}(\mathcal{T}_{0_x}(\xi^{\ver})) 
		= \xi^{\ver},
	\end{align*}
	where $\mathcal{P}^{\gamma}_{t_1, t_2}\colon T_{\gamma(t_1)}M\to T_{\gamma(t_2)}M$ denotes the parallel transport along $\gamma$. Therefore, we can identify $\xi$ as well as the vector $(\textup{d}\mathcal{R}_{(x, v)})_{0_{(x, v)}}(\xi)$ with $((x, v), \xi^{\hor}, \xi^{\ver})$ which concludes the proof.
\end{proof}

\section{Retraction-based regression on the $p$-norm unit sphere}
\label{sec:numerics}

In this section, we apply the general framework developed in \cref{sec:main} for regression problems to the $p$-norm sphere. In the following, we derive specific expressions of \eqref{map: auxillary maps} for the $p$-norm sphere and specify the chosen retractions. By the end of this section, we will present three numerical experiments. For a more concise notation we introduce the Hadamard product by $\odot\colon\RR^n\times\RR^n\to \RR^n,\ (v, w)\mapsto (v_1\cdot w_1, \ldots, v_n\cdot w_n)$ as well as the sign function by $\sgn\colon\RR\to\{-1, 0, 1\}$. In addition, all functions of the type $f\colon\RR\to X\subset \RR$ are extended to $\RR^n$ by applying them component-wise, e.g., $f(x) = (f(x_1), \ldots f(x_n))$ for $x = (x_1,\dots,x_n)\in\RR^n$.

The $(n-1)$-dimensional $p$-norm unit sphere for $p\in(1, \infty)$ is defined by
\begin{align*}
	S_p^{n-1}
	\coloneqq\set{x\in\RR^n}{\norm{x}_p\coloneqq \sqrt[p]{\sum\limits_{i = 1}^n |x_i|^p} = 1}
	= \phi^{-1}\left(\{1\}\right),
\end{align*}
where $\phi\colon\RR^n\to \RR,\ x\mapsto \norm{x}_p^p$ and $\phi^{-1}(A)$ denotes the inverse image of $\phi$ of the set $A$. As stated in the proof \cite[Thm.~1]{sato2023} $\phi$ is a defining function for $S_p^{n-1}, p\in(1, \infty)$. Indeed, $S_p^{n-1}$ is a $(n - 1)$-dimensional $C^r$ embedded Riemannian submanifold of $\RR^n$. The regularity of this manifold, denoted by the index $r$, depends on whether $p$ is even, odd or not an integer. 

The the tangent space of $S_p^{n-1}$ at $x\in\phi^{-1}(\{1\})$ can be described by
\begin{align*}
	T_xS_p^{n-1}
	= \Ker(\textup{d}\phi_x) 
	= \set{y\in\RR^n}{\left<\grad\phi(x), y\right>  = 0} 
	= \set{y\in\RR^n}{\left<\sgn(x)\odot|x|^{p-1}, y\right>  = 0},
\end{align*}
where $\grad \phi(x) = p\cdot \sgn(x)\odot |x|^{p-1}$ (c.f., \cite[p.~900]{sato2023}). Therefore, the map $\psi\colon\RR^n\times\RR^n\to\RR^2,\ (x, v)\mapsto (\phi(x), \textup{d}\phi_x(v))$ is a defining function for the tangent bundle so that we have
\begin{align*}
	TS_p^{n-1}
	= \set{(x, v)\in \RR^n\times \RR^n}{\phi(x) = 1 \;\ \text{ and }\; \textup{d}\phi_x(v) = 0}
	= \psi^{-1}(\{(1, 0)\}).
\end{align*}
As shown in~\cite[Thm.~3.43, p.~37]{boumal2023} the tangent bundle is a $2(n - 1)$-dimensional $C^{r-1}$ embedded Riemannian submanifold of $\RR^n\times\RR^n$. It should be noted that $r - 1\geq 0$ and therefore well defined. 

There are several retractions defined on the sphere, but from now on we focus on the retraction by normalization. 
The definition of this retraction is not intrinsically as the vector space structure of the ambient space $\RR^n$ is used. 
The retraction is defined  by
\begin{align}
	\label{map: retr by normalization}
	R\colon TS_p^{n - 1}\to S_p^{n-1},\ (x, v)\mapsto \frac{x + v}{||x + v||_p} 
\end{align}
for $p\in(1, \infty)$.
According to \cite[Prop.~3]{sato2023} and \cite[Prop.~3.2, p.~8]{absil2012} this mapping fulfils both conditions for being a retraction on $S_p^{n-1}$.
Additionally, its point-wise inverse, computed in \cite[Prop.~4]{sato2023}, is given by
\begin{align*}
	R^{-1}_x\colon D_x\to T_xS_p^{n - 1},\
	y\mapsto p\frac{y}{\left<\grad\phi(x), y\right>} - x,
	\quad 
	x\in S_p^{n-1}.
\end{align*}
where the set $D_x = \set{y\in S_p^{n - 1}}{\left<\grad\phi(x), y\right> > 0}$ describes an open half-space of the sphere. To calculate the gradient of the objective function $\mathcal{F}_R$ given in~\eqref{map: objective function} we use the formula~\eqref{eq: embedded gradient objective fct}. For the following computations we adopt the notation of \eqref{map: auxillary maps} for $h_i$ and \eqref{eq: embedded gradient objective fct} for $\psi$. Thus, we have to compute explicit expressions for $J_{h_i}$ and $J_{\psi}$. Therefore, let $\mathcal{D} = \set{(t_i, y_i)\in \RR\times S_p^{n-1}}{i = 1, \ldots, m}$ be a predefined data set so that for each $i\in\{1, \ldots, m\}$ the map $h_i$, generally defined in \eqref{map: auxillary maps}, is given by
\begin{align*}
	&h_i\colon V_{y_i}\to T_{y_i}S_p^{n-1},\ 
	(x, v)\mapsto R_{y_i}^{-1}(R_x(t_iv))
	= p\frac{x + t_iv}{\left<\grad\phi(y_i), x + t_iv\right>} - y_i,
\end{align*}
where $V_{y_i} = \set{(x, v)\in TM}{\left<\grad\phi(y_i), x + t_iv\right> > 0}$. From~\cref{prop:well-posed}, we recall that the optimization domain $V$ takes the form $V = \bigcap_{i= 1, \ldots, m} V_{y_i}$.
Since the codomain of $\phi$ is $\RR$, the matrix $A_{\phi}(x)$ is just the Hessian of $\phi$ given by $\Hess\phi(x) = p(p-1) \diag(|x_i|^{p-2})$ in $x\in\RR^n$. Thus, by using the standard coordinate frame $(\partial_{x_i}, \partial_{v_i})$ we get for each $(x, v)\in V$
\begin{align*}
	J_{\psi}(x, v)
	= \begin{pmatrix}
		\grad\phi(x)^{\top} & 0 \\
		v\cdot \Hess\phi(x) & \grad\phi(x)^{\top}
	\end{pmatrix}
	= p \begin{pmatrix}
		\sgn(x)\odot |x|^{p-1} & 0 \\
		(p-1) v\odot |x|^{p-2} & \sgn(x)\odot |x|^{p-1}
	\end{pmatrix}.
\end{align*}
Additionally, for $(x,v)\in V$ it holds
\begin{align*}
	\partial_{x_j}g_i(x, v)
	&= p \frac{\left<\grad\phi(y_i), x + t_iv\right> e_j - \left<\grad\phi(y_i), e_j\right> (x + t_iv)}{\left<\grad\phi(y_i), x + t_iv\right>^2}, \\
	\partial_{v_j}g_i(x, v)
	&= t_i \partial_{x_j}g_i(x, v),
\end{align*}
where $e_j$ denotes the $j$-th standard basis vector in $\RR^n$.

As  previously mentioned, the use of the Riemannian gradient descent algorithm to solve the regression problem~\eqref{eq:regression_main}, we must specify a retraction on $T^2S_p^{n-1}$, that is a mapping $\mathcal{R}\colon T^2S_p^{n-1}\to TS_p^{n-1}$. In virtue of \cref{theorem: retraction tangent bundle} it is sufficient to specify a vector transport. In our numerical experiments, we use the vector transport defined by differentiating the retraction $R$:
\begin{align}
	\label{map: vector transport}
	\mathcal{T}\colon TS_p^{n-1}\times_{S_p^{n-1}} TS_p^{n-1} \to TS_p^{n-1}, \
	(\xi_x, \eta_x)\mapsto \mathcal{T}_{\eta_x}(\xi_x)
	\coloneqq 
	(dR_x)_{\eta_x}(\xi_x).
\end{align}
Here, we can provide an explicit formula for \eqref{map: vector transport} since
\begin{align*}
	(dR_x)_{\eta_x}(\xi_x) = \frac{1}{\norm{x+\eta_x}_p}\left(\xi_x - \left<\grad \norm{\cdot}_p (x + \eta_x), \xi_x\right> R_x(\eta_x)\right),
\end{align*}
which degenerates to $\frac{1}{\norm{x + \eta_x}}P_{R_x(\eta_x)}(\xi_x)$ for $p=2$, where $P_{R_x(\eta_x)}$ denotes the orthogonal projection onto $T_{R_x(\eta_x)}S_2^{n-1}$ and $\norm{\cdot}$ the Euclidean norm, i.e., for $p=2$ \eqref{map: vector transport} reduces to
\begin{align}
	\label{map: vector transport S_2}
	\mathcal{T}\colon TS_2^{n-1}\times_{S_2^{n-1}} TS_2^{n-1} \to TS_2^{n-1}, \
	(\xi_x, \eta_x)\mapsto \mathcal{T}_{\eta_x}(\xi_x)=\frac{1}{\norm{x + \eta_x}}P_{R_x(\eta_x)}(\xi_x).
\end{align}
See \cite[Ex.~8.1.4, p.~172]{absil2008} for a proof that \eqref{map: vector transport} is indeed a vector transport.

\paragraph{Numerical experiments.}

The numerical experiments presented in this section are aimed to analyze the viability of the proposed framework as a proof of concept. As we aim to test whether the optimization algorithm converges to a solution of the problem, we generate synthetic data from a known solution. Additionally, to improve the results visualization, we restrict the problem to the two-dimensional $2$-sphere. However, some of the incoming constructions are kept general for $S_p^{n-1}$. All numerical experiments were implemented in Python using the package Pymanopt \cite{townsend2016}. From all the available optimization algorithms, we employ the Riemannian gradient descent combined with an Armijo backtracking line search strategy.

The independent variable samples $t_1, \ldots, t_m$ are generated by dividing a predefined interval $[a, b]\subset \RR$ into $m-1$ equal parts. The corresponding manifold-valued observations are obtained by evaluating our model $c_{R, (x^*, v^*)}$ for fixed parameter values $(x^*, v^*)\in TS_p^{n-1}$ on each time point and adding Gaussian perturbations to the resulting outputs. Therefore, the $i$-th data point on the sphere takes the form $y_i = R\left(\hat{y}_i, \varepsilon_i\right)$, where $\varepsilon_i = P_{\hat{y}_i}(z_i)$ denotes a Gaussian perturbation taking values in the tangent space at the model prediction  $\hat{y}_i\coloneqq c_{R, (x^*, v^*)}(t_i)\in S_p^{n-1}$. Here, $z_i$ is the $i$-th observation of the standard Gaussian random vector $Z \sim \mathcal{N}(0_{\RR^n}, \mu E_n)$ in $\RR^n$. The choice of $\mu$ requires particular care, since excessively large perturbations may lead to an empty feasible optimization domain due to the restricted domain of the inverse retraction. The specific values of $x^*$ and $v^*$ for generating our data is specified in \cref{tab:results}.

We present the obtained solutions $(x,v)\in TS_p^{n-1}$ and their corresponding trajectories together with the associated cost functional and gradient norm histories for three sets of synthetic data on the two-dimensional $2$-sphere, which we denote by $S^2 = S_2^2$ for simplicity. The experiments differ solely in the variance $\mu\geq 0$ of the perturbing noise. \cref{tab:parameters} shows the value of the parameters used throughout all experiments. 
\begin{table}[h]
	\centering
	\begin{tabular}{|c|c|}
		\hline
		data set size &  $100$\\
		\hline
		random seed & $5$ \\ 
		\hline
		time interval & $(-3, 3)$\\
		\hline
		retraction & by normalization (\ref{map: retr by normalization}) \\
		\hline
		vector transport & by differentiated retraction (\ref{map: vector transport S_2})\\
		\hline
		$x^*$ & $(-0.51, -0.85, -0.16)$\\
		\hline
		$v^*$ & $(-0.78, 0.53, -0.33)$ \\
		\hline
	\end{tabular}
	\caption{Parameter used in all three numerical experiments where $(x^*, v^*)$ is rounded to two decimal places.}
	\label{tab:parameters}
\end{table}
The final results as well as the real solutions are presented in \cref{tab:results}.

In all three cases, the optimization process stopped after less than $30$ iterations satisfying that the gradient norm was below $10^{-6}$. The plot of the gradient norm suggests that the convergence rate is linear agreeing with the theory. The obtained solutions are very close to the real one differing only after the second decimal place.

\begin{table}[h]
	\centering
	\begin{tabular}{|c|c|c|c|}
		\hline
		& \multicolumn{3}{c|}{Noise Level $\mu$}\\
		\hline
		& $0$ & $0.08$ & $0.205$\\
		\hline
		$x$ & $(-0.51, -0.85, -0.16)$ & $(-0.51, -0.84, -0.16)$ & $(-0.52, -0.84, -0.16)$\\
		\hline
		$v$ & $(-0.78, 0.53, -0.33)$ & $(-0.78, 0.54, -0.31)$ & $(-0.79, 0.54, -0.29)$\\
		\hline
		iterations & $29$ & $27$ & $29$\\
		\hline
		cost & $5.46\cdot 10^{-15}$ & $0.63$ & $4.11$\\
		\hline
		gradient norm & $6.33\cdot 10^{-7}$ & $4.44\cdot 10^{-7}$ & $3.89\cdot 10^{-7}$\\
		\hline
	\end{tabular}
	\caption{Final results of all three numerical experiments which are rounded to two decimal places.}
	\label{tab:results}
\end{table}

\begin{figure}[ht]
	\centering
	\begin{subfigure}{0.32\textwidth}
		\centering
		\includegraphics[width=\linewidth]{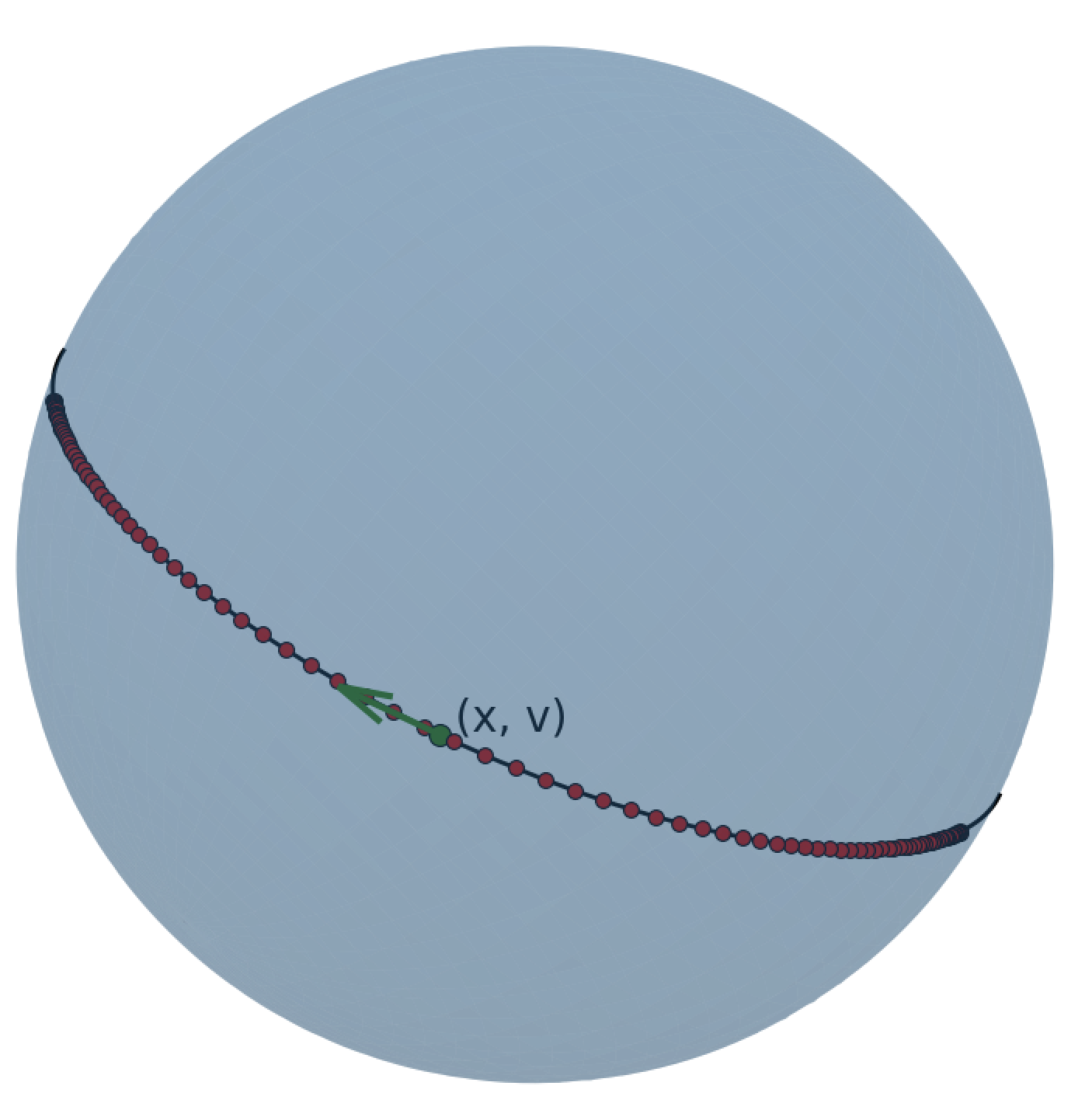}
		\label{fig:reg_line_no_noise}
	\end{subfigure}
	\hfill
	\begin{subfigure}{0.32\textwidth}
		\centering
		\includegraphics[width=\linewidth]{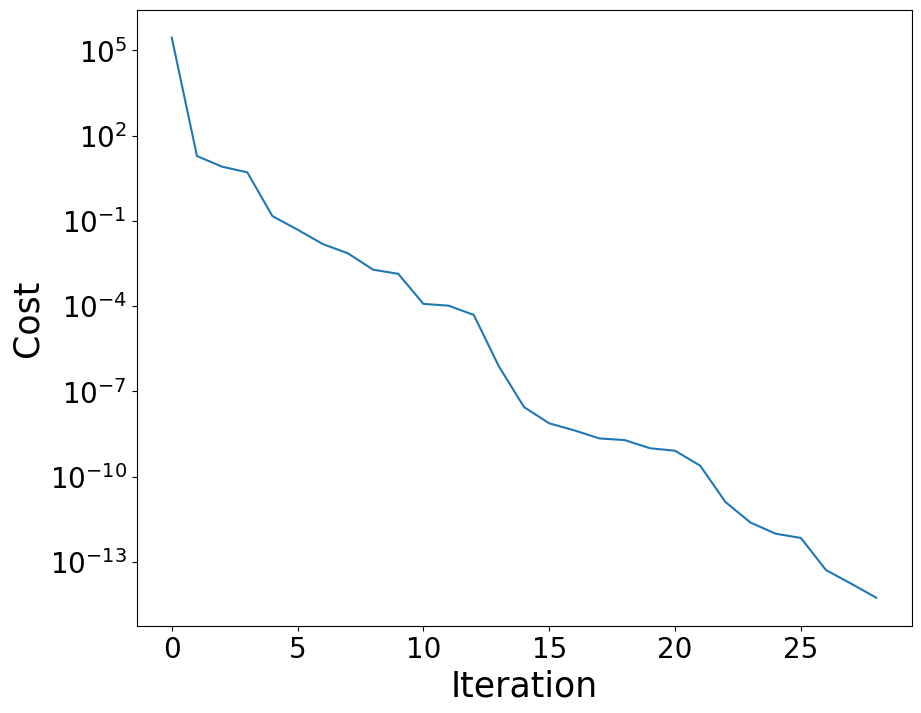}
		\label{fig:cost_no_noise}
	\end{subfigure}
	\hfill
	\begin{subfigure}{0.32\textwidth}
		\centering
		\includegraphics[width=\linewidth]{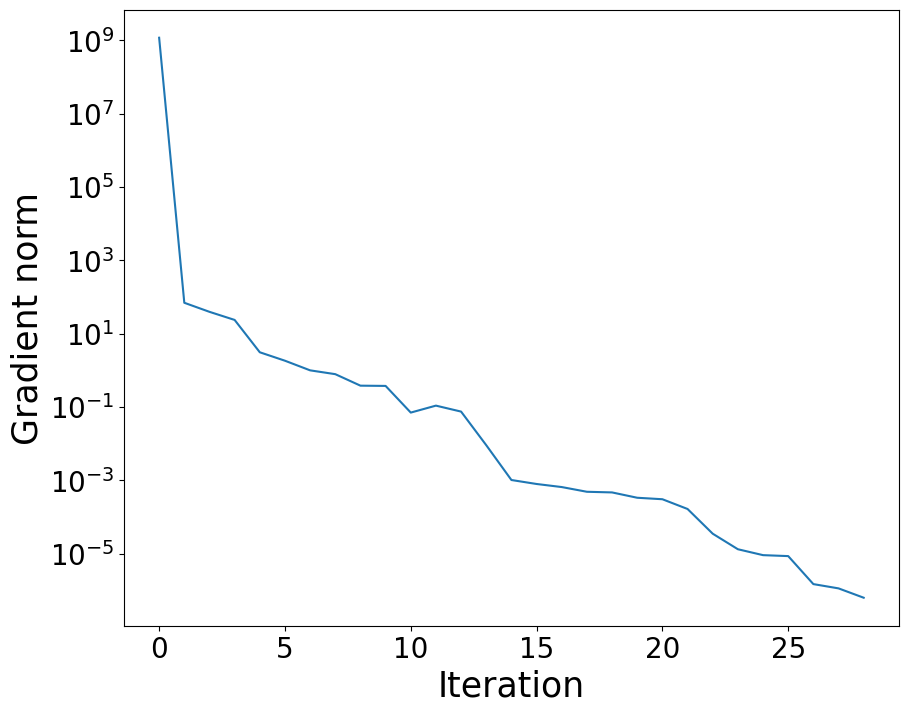}
		\label{fig:gradient_norm_no_noise}
	\end{subfigure}
	\caption{Results obtained for $\mu = 0$. In the left panel the red dots represent the data, and the black solid line the associated regression line with $(x,v)$ depicted as green dot and arrow. The middle and right panels show the cost and gradient norm histories, respectively.}
\end{figure}
\begin{figure}[ht]
	\centering
	\begin{subfigure}{0.32\textwidth}
		\centering
		\includegraphics[width=\linewidth]{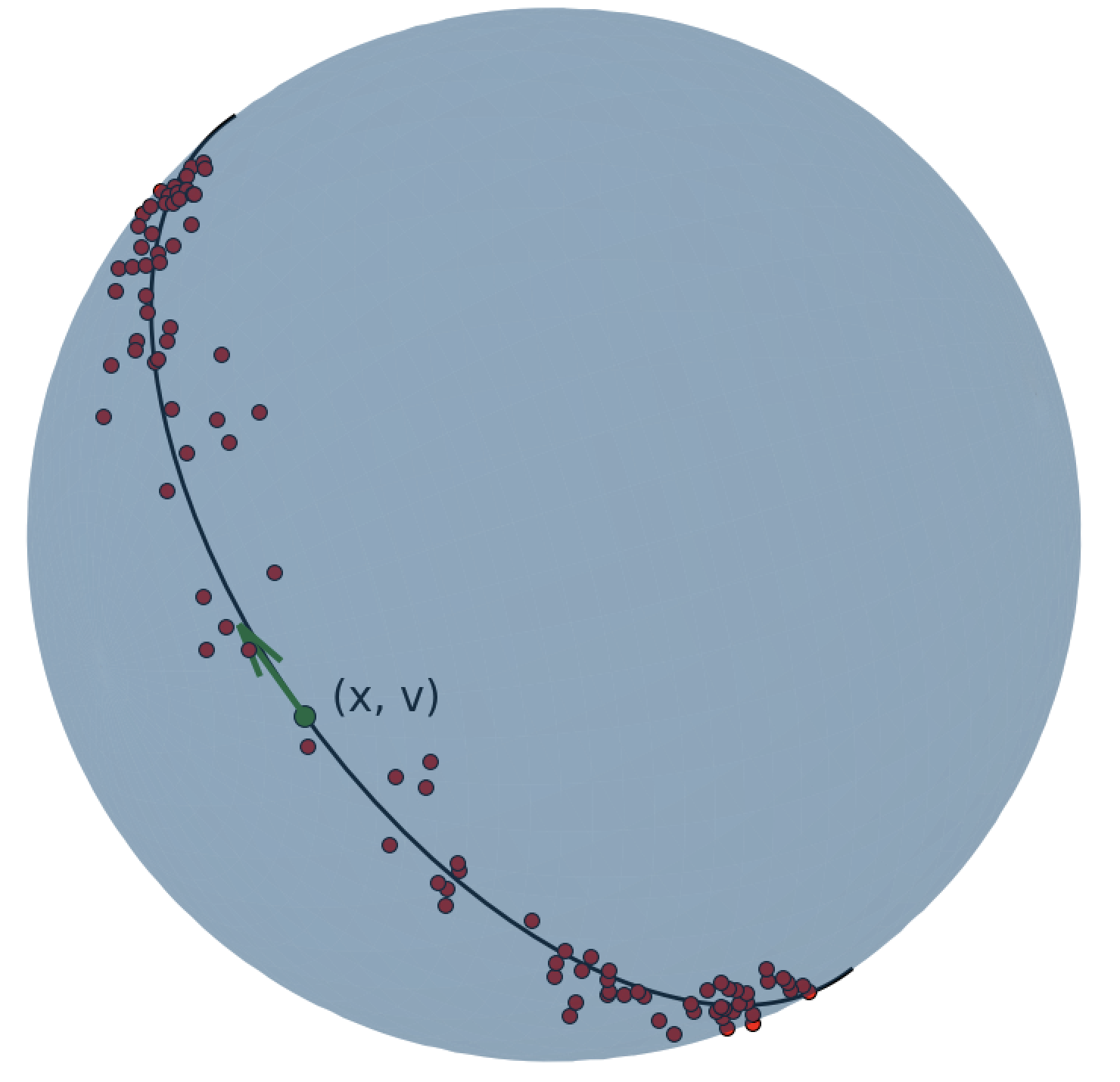}
		\label{fig:reg_line_small_noise}
	\end{subfigure}
	\hfill
	\begin{subfigure}{0.32\textwidth}
		\centering
		\includegraphics[width=\linewidth]{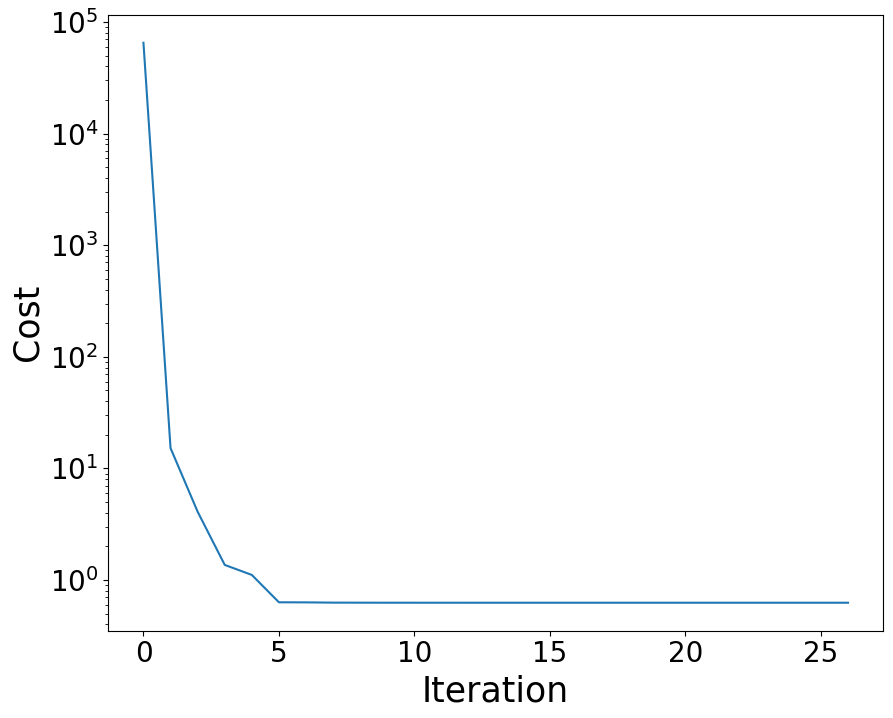}
		\label{fig:cost_small_noise}
	\end{subfigure}
	\hfill
	\begin{subfigure}{0.32\textwidth}
		\centering
		\includegraphics[width=\linewidth]{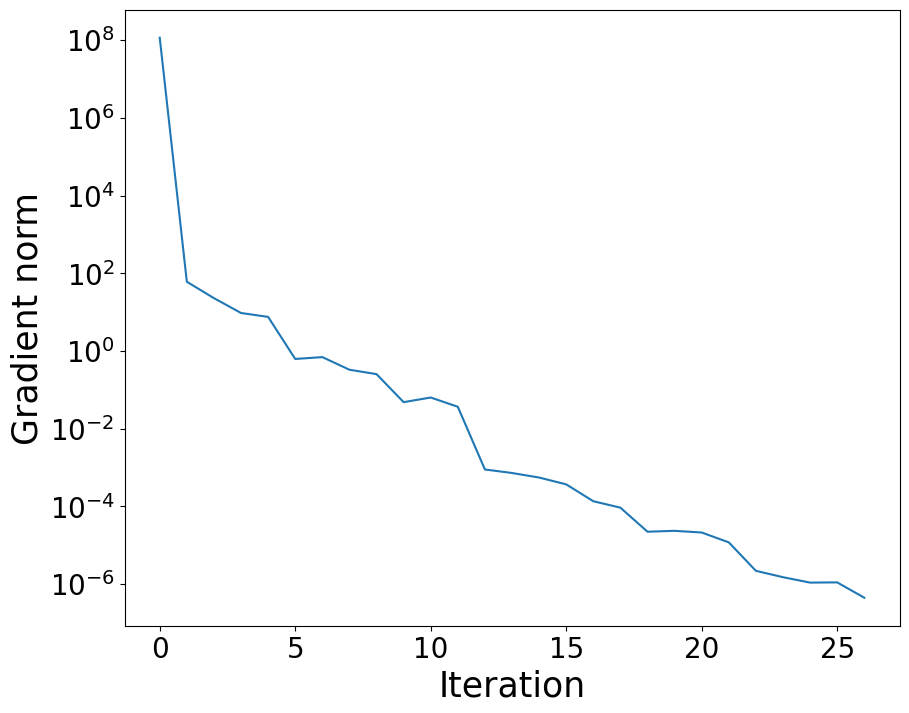}
		\label{fig:gradient_norm_small_noise}
	\end{subfigure}
	\caption{Results obtained for $\mu = 0.08$. In the left panel the red dots represent the data, and the black solid line the associated regression line with $(x,v)$ depicted as green dot and arrow. The middle and right panels show the cost and gradient norm histories, respectively.}
\end{figure}
\begin{figure}[ht]
	\centering
	\begin{subfigure}{0.32\textwidth}
		\centering
		\includegraphics[width=\linewidth]{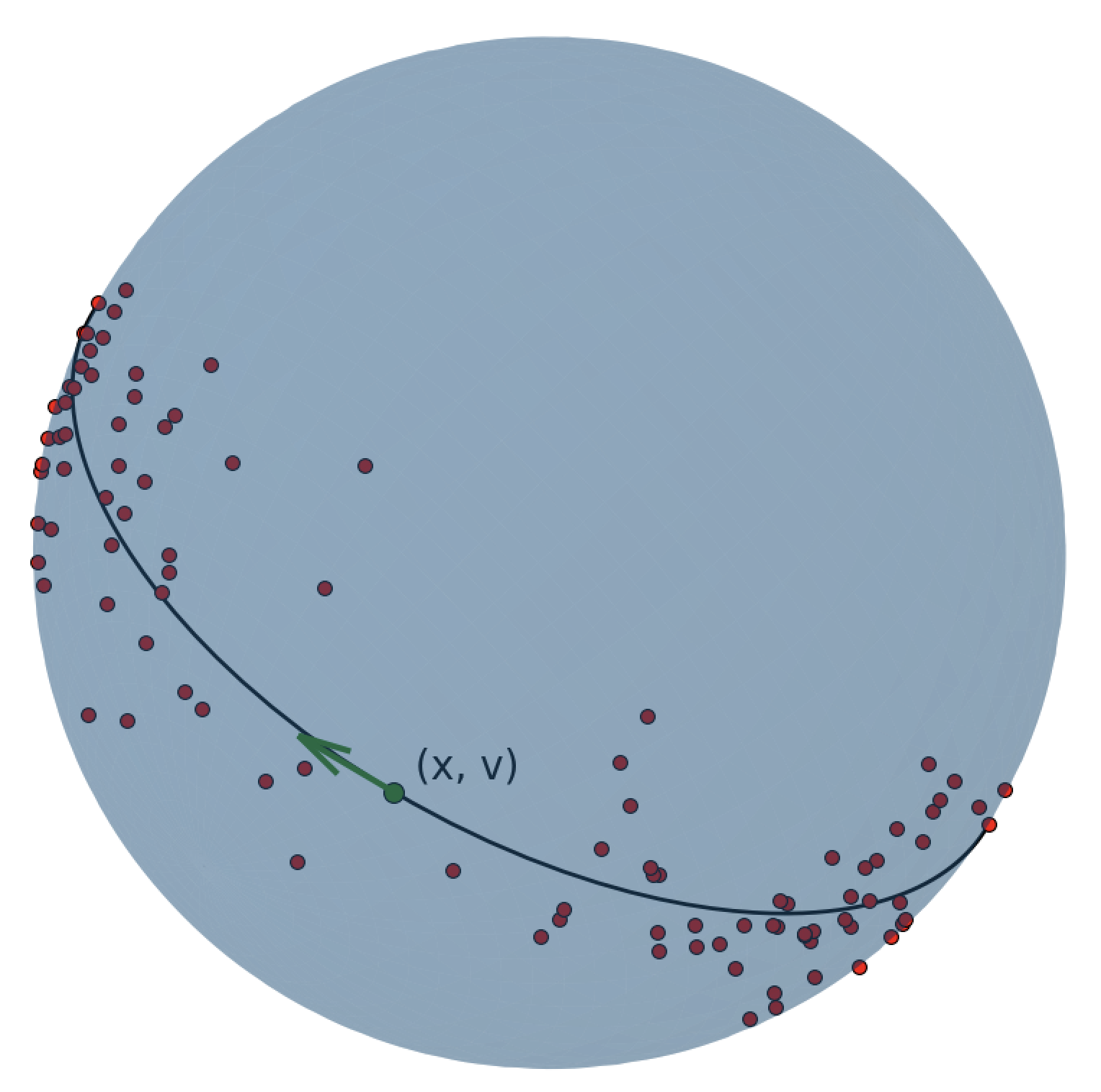}
		\label{fig:reg_line_high_noise}
	\end{subfigure}
	\hfill
	\begin{subfigure}{0.32\textwidth}
		\centering
		\includegraphics[width=\linewidth]{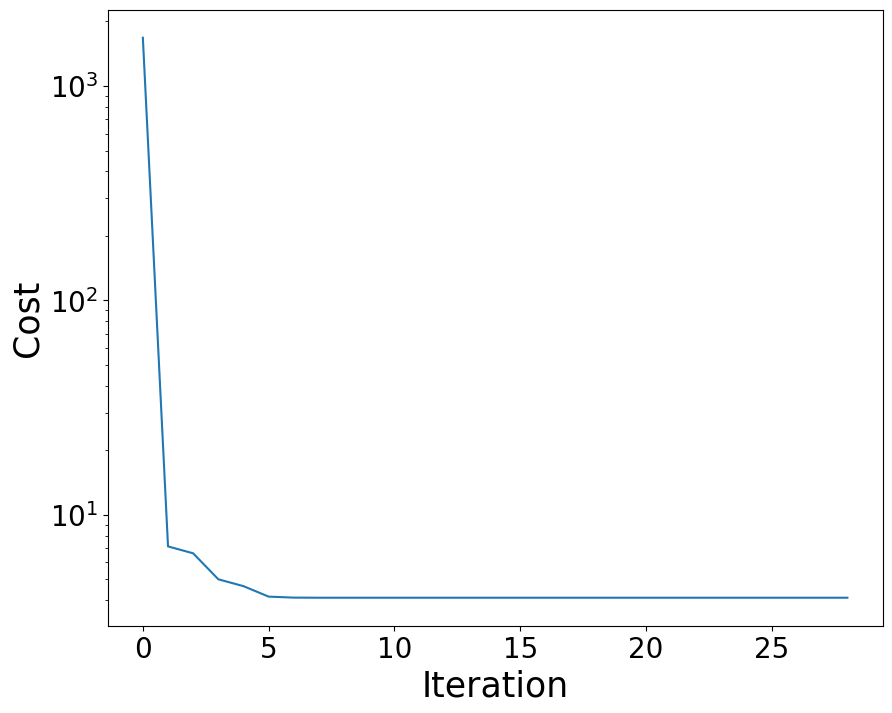}
		\label{fig:cost_high_noise}
	\end{subfigure}
	\hfill
	\begin{subfigure}{0.32\textwidth}
		\centering
		\includegraphics[width=\linewidth]{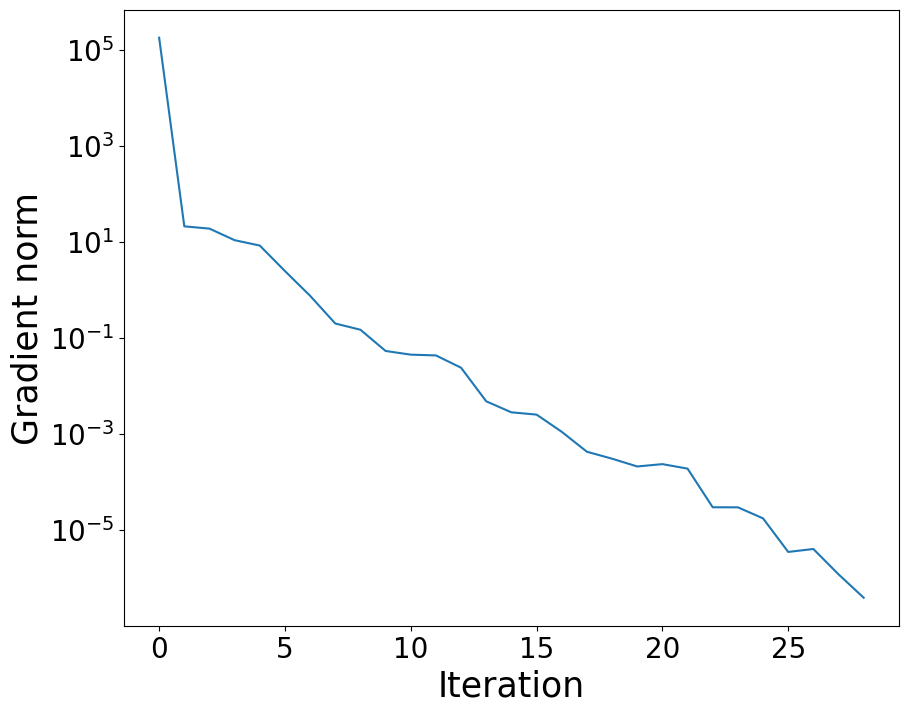}
		\label{fig:gradient_norm_high_noise}
	\end{subfigure}
	\caption{Results obtained for $\mu = 0.205$. In the left panel the red dots represent the data, and the black solid line the associated regression line with $(x,v)$ depicted as green dot and arrow. The middle and right panels show the cost and gradient norm histories, respectively.}
\end{figure}

\section{Conclusion}
\label{sec:Conc}
We presented a novel approach for regression on Riemannian manifolds by replacing the Riemannian exponential map with computationally cheaper retractions and the corresponding retraction-based distance. To guarantee the well-posedness of the objective function, we restrict the optimization domain to an open subset of the tangent bundle. It is important to mention that the optimization domain obtained in \cref{prop:well-posed} could be empty if the data is too scattered on the manifold $M$. Additionally, we have decided to delay the proof of existence of solutions for this problem for a later publication. Furthermore, we provided a general formula for the gradient of the objective function which coincides with that one given in \cite{fletcher2013} in the special case of geodesic regression. For embedded Riemannian submanifolds of a Euclidean space, we additionally provided a specific version of the gradient. Finally, a method for constructing a retraction on the tangent bundle was proven.

As a proof of concept, we applied the proposed framework to the $(n-1)$-dimensional, $p$-norm sphere together with the retraction by normalization. Showing that the restriction of the optimization domain albeit being restrictive it does hold in this case. It is worth highlighting that the choice of a suitable retraction remains a challenge. As the retraction imposes local conditions, it provides limited control over its global behaviour.

We investigated numerically the resulting optimization problem by generating three different synthetic data, where the solutions were known in advance. The experiments indicate favourable convergence behaviour of the proposed method. Although geodesics on the sphere are available in closed form and are therefore computationally tractable, the $(n-1)$-dimensional, $p$-norm sphere provide a suitable setting for validating the proposed framework. A major advantage of the presented approach is its potential applicability to a broader class of Riemannian manifolds for which the computation of geodesics is expensive or infeasible. Future work will focus on identifying classes of retractions that are particularly suitable for regression problems, extending the theoretical analysis to more general manifolds, and conducting further numerical experiments. In particular, we aim to apply the proposed framework to shape spaces, where the computation of geodesics is often prohibitively expensive. Specific applications include, for example, the shape space $B_e$ equipped with the Steklov--Poincar\'e metric \cite{schulz2016} or the manifold of planar triangular meshes \cite{herzog2022a, herzog2024}.

\bibliographystyle{plain}
\bibliography{citations}

\end{document}